\documentclass{amsart}

\usepackage{todonotes}
\usepackage{enumerate}
\usepackage{amssymb}
\usepackage{amsmath,amscd}
\usepackage{amsthm}
\usepackage{graphicx}
 \usepackage{verbatim}
 \usepackage{stmaryrd}

\newcommand{\R}{\mathbb{R}}

\newcommand{\In}{\subset}

\newcommand{\bil}[1]{\left< #1\right>}

\newcommand{\Ad}{\operatorname{Ad}}

\newcommand{\II}{I\!I}

\newcommand{\g}{\mathfrak{g}}
\newcommand{\h}{\mathfrak{h}}
\newcommand{\m}{\mathfrak{m}}

\newtheorem{theorem}{Theorem}

\newtheorem*{corollary*}{Corollary}
\newtheorem{lemma}[theorem]{Lemma}
\newtheorem{proposition}[theorem]{Proposition}

\newtheorem*{maintheorem}{ Main Theorem}

\theoremstyle{definition}
\newtheorem{definition}[theorem]{Definition}

\theoremstyle{remark}

\newtheorem{example}[theorem]{Example}

\title[Virtual immersions and symmetric spaces]{Virtual immersions and a characterization of symmetric spaces}

\author[R.~Mendes]{Ricardo A. E. Mendes$^* \dag$}
\address{University of Cologne, Germany}
\email{rmendes@gmail.com}

\author[M.~Radeschi]{Marco Radeschi}
\address{University of Notre Dame, US}
\email{mradesch@nd.edu}

\thanks{$\dag$ received support from DFG ME 4801/1-1}

\subjclass[2010]{49Q05, 53A10, 53C35}
\keywords{Symmetric space, Isometric immersion}

\begin{document}

\maketitle
\begin{abstract}
We define virtual immersions, as a generalization of isometric immersions in a pseudo-Riemannian vector space. We show that  virtual immersions possess a second fundamental form, which is in general not symmetric. We prove that a manifold admits a virtual immersion with skew symmetric second fundamental form, if and only if it is a symmetric space, and in this case the virtual immersion is essentially unique.
\end{abstract}

\section{Introduction}

Often in Riemannian geometry, one needs to embed a Riemannian manifold into Euclidean or Lorentzian space. In this paper we introduce a generalized, and more ``intrinsic'' version of such embeddings, and utilize them to give a new characterization of symmetric spaces.

Given a Riemannian manifold $M$ and an isometric immersion $\phi:M\to V$ into a vector space $(V,\bil{,})$ endowed with a non-degenerate symmetric bilinear form (a \emph{pseudo-Euclidean vector space}), then the pullback $\phi^*TV$ is a trivial vector bundle over $M$, the differential $\phi_*$ defines an immersion $\phi_*:TM\to \phi^*TV$, and the classical results on isometric immersions show that the canonical (flat) connection on $\phi^*TV$ induces, by projecting onto $TM$, the Levi Civita connection on $M$. We use this properties to define a \emph{virtual immersion} of a Riemannian manifold $M$, as a flat bundle $M\times V$, with $V$ a pseudo-Euclidean vector space, together with an isometric embedding $TM\to M\times V$ such that the flat connection on $M\times V$ induces the Levi Civita connection on $M$ (see Definition \ref{D:genimm} for an equivalent definition).

It turns out that, just like the usual isometric immersions, one can define a second fundamental form, but unlike the usual setting this is in general \emph{not} symmetric. As a matter of fact, it can be easily shown that a virtual immersion is (locally) induced by an isometric immersion, if and only if the second fundamental form is symmetric.

In \cite{MR17}, we first introduced virtual immersions with $V$ Euclidean (rather than pseudo-Euclidean) in the context of verifying, for certain compact symmetric spaces, a conjecture of Marques-Neves-Schoen about the index of closed minimal hypersurfaces. In that same paper, it was proved that, when $V$ has a Euclidean metric, virtual immersions with skew-symmetric second fundamental form exist only on compact symmetric spaces (cf. \cite{MR17}, Theorem B).

The main result of this paper is to extend the classification of virtual immersions with skew symmetric second fundamental form to the more general case in which the metric on $V$ is pseudo-Euclidean:

\begin{maintheorem}
\label{MT:classification}
Let $(M,g)$ be a Riemannian manifold. Then $M$ admits a virtual immersion $\Omega$  with skew-symmetric second fundamental form if and only if it is a symmetric space. In this case, $\Omega$ is essentially unique.
\end{maintheorem}

Virtual immersions, in other words, provide a bundle-theoretic characterization of symmetric spaces, although we expect them to have independent interest on more general spaces.

The paper is organized as follows: in Section \ref{S:intro} we define virtual immersions and their second fundamental form, and establish their fundamental equations. In Section \ref{S:virtual-imm} we prove the ``if'' part of the Main Theorem, producing a virtual immersion with skew-symmetric second fundamental form on any symmetric space. In Section \ref{S:rigidity} we prove the ``only if'' part of the Main Theorem, showing that a virtual immersion with skew-symmetric second fundamental form forces the manifold to be a symmetric space. In this last section, we also glue the pieces together, and prove the Main Theorem.

{\bf Convention:} We will denote by $R$ the curvature tensor, and follow the sign convention 
$ R(X,Y)Z= \nabla_Y\nabla_X Z - \nabla_X \nabla_Y Z + \nabla_{[X,Y]}Z $.

\section{Virtual immersions}\label{S:intro}

Let $(M,g)$ be a Riemannian manifold, and let $(V,\bil{,})$ denote a real vector space endowed with a nondegenerate, symmetric bilinear form. We call such $(V, \bil{,})$ a \emph{pseudo-Euclidean vector space}. A $V$-valued virtual immersion of $M$ is, roughly speaking, an immersion of $TM$ into the trivial bundle $M\times V$, such that the natural flat connection on $M\times V$ induces the Levi-Civita connection of $M$. Such objects generalize isometric immersions of Riemannian manifolds in Lorentzian space.

Although this is the idea behind virtual immersions, we introduce such structures in a different way, more convenient for computations --- see Proposition \ref{P:equivalentdefinition} for a proof that the two definitions coincide.

\begin{definition}
\label{D:genimm}
Let $(M,g)$ be a Riemannian manifold, and $(V,\bil{,})$ a finite-dimensional, pseudo-Euclidean real vector space. Let $\Omega$ be a $V$-valued one-form on $M$. We say $\Omega$ is a \emph{virtual immersion} if the following two conditions are satisfied:
\begin{enumerate}[a)]
\item $\left<\Omega(X),\Omega(Y)\right>=g(X,Y)$ for every $p\in M$, and every $X,Y\in T_pM$.
\item $\left<d\Omega(X,Y),\Omega(Z)\right>=0$ for every $p\in M$, and every $X,Y, Z\in T_pM$.
\end{enumerate}
We say two virtual immersions $\Omega_i:TM\to V_i$, $i=1,2$ are equivalent if there is a linear isometry $(V_1,\bil{,}_1 )\to (V_2,\bil{,}_2 )$ making the obvious diagram commute.
\end{definition}

Letting $\pi:TM\to M$ denote the foot-point projection, any virtual immersion $\Omega:TM\to V$ induces a vector bundle homomorphism $(\pi, \Omega):TM\to M\times V$. By condition $(a)$ in the definition, this map is an isometric immersion of (\emph{pseudo-Euclidean}) vector bundles.

Fixing $p\in M$, denote by $\Omega_p:T_pM\to V$ the restriction of $\Omega$ to $T_pM$. Since $\Omega_p$ is an isometric immersion, the space $T_pM$ can be identified with its image, which we will still denote $T_pM$. Moreover, since the metric on $T_pM$ is positive definite, its orthogonal complement $\nu_pM:=(T_pM)^\perp\subset V$ is transverse to $T_pM$ and thus $V$ splits orthogonally as $V=T_pM\oplus \nu_pM$. This yields the orthogonal decomposition $M\times V=TM\oplus \nu M$. Given $(p,X)\in M\times V$, we shall write $X=X^T+X^\perp$ for the decomposition into the tangent and normal parts.

The natural flat connection $D$ on $M\times V$ induces a connection $D^T$ (respectively  $D^\perp$) on $TM$ (resp. $\nu M$), given by $D^T_X Y=(D_X Y)^T$ (resp. $D^\perp_X \eta=(D_X \eta)^\perp$). Here $X,Y$ are vector fields on $M$, while $\eta$ is a section of the normal bundle.

\begin{proposition}
\label{P:equivalentdefinition}
Let $\Omega$ be a $V$-valued one-form on $M$  satisfying condition (a) in Definition \ref{D:genimm}. Then, $\Omega$ is a virtual immersion if and only if the flat connection $D$ on $M\times V$ satisfies $D^T=\nabla$, where $\nabla$ denotes the Levi Civita connection on $M$.
\end{proposition}
\begin{proof}
Since $\Omega$ already satisfies condition $(a)$, it is a virtual immersion if and only if condition $(b)$ holds as well, that is, $d\Omega(X,Y)^T=0$ for every point $p$ and every $X,Y\in T_pM$. Recall that
\begin{equation}
\label{E:dOmega}
 d\Omega(X,Y)=
D_X Y-D_Y X-[X,Y] \end{equation}
so that taking the tangent part yields
\[  d\Omega(X,Y)^T= D^T_X Y -D^T_Y X - [X,Y] .\]
Condition (a) implies that $D^T$ is compatible with the metric $g$, and by the above formula condition (b) is equivalent to $D^T$ being torsion-free. Since these two properties characterize the Levi Civita connection, the result follows.
\end{proof}

Given a virtual immersion $\Omega:TM\to V$ and a group $\Gamma$ of isometries of $M$, we say that $\Omega$ is $\Gamma$-invariant if for every $\gamma\in \Gamma$, $\Omega\circ d\gamma=\Omega$, where $d\gamma:TM\to TM$ denotes the differential of $M$. The following result is straightforward:

\begin{lemma}
Let $\Omega:TM\to V$ be a virtual immersion, and let $\pi:\tilde M\to M$ denote a covering. Then $\pi^*\Omega=\Omega\circ d\pi:T\tilde M\to V$ is a virtual immersion, which is invariant under the deck group of $\tilde M\to M$. Conversely, if $\Omega:TM\to V$ is invariant under a group $\Gamma$ acting freely on $M$ by isometries, and $\pi: M\to M'=M/\Gamma$ denotes the quotient, then $\Omega$ descends to a virtual immersion $\Omega':TM'\to V$ such that $\Omega=\pi^*\Omega'$.
\end{lemma}


Given a virtual immersion $\Omega:TM\to V$ and a linear isometric immersion $\iota:V\to W$, there is an induced virtual immersion $\iota\circ \Omega:TM\to W$. We want to rule out these trivial extensions.
\begin{definition}
A virtual immersion $\Omega:TM\to V$ is called \emph{full} if the image of $\Omega$ spans $V$.
\end{definition}

For any virtual immersion $\Omega:TM\to W$, defining the subspace $V=\textrm{span}(\Omega(TM))$ and letting $\iota:V\to W$ denote the inclusion, one obtains the following:

\begin{lemma}
Given any virtual immersion $\Omega:TM\to W$ there exists a full immersion $\Omega':TM\to V$ and a linear isometric immersion $\iota:V\to W$ such that $\Omega=\iota\circ \Omega'$.
\end{lemma}

Given a virtual immersion, one can define a second fundamental form and shape operator.

\begin{definition}
Let $\Omega$ be a $V$-valued virtual immersion, $X,Y$ be smooth vector fields on $M$, and $\eta$ a smooth section of $\nu M$. Define the \emph{second fundamental form} of $\Omega$ by
\[ \II:TM\times TM\to \nu M,\qquad\II(X,Y)=
(D_X Y)^\perp=D_X(\Omega(Y))-\Omega(\nabla_XY)
 \]
 and the \emph{shape operator} in the direction of a normal vector $\eta$ by
 \[
 S_\nu:TM\to TM,\qquad S_\eta(X)=-(D_X \eta)^T.
 \]
\end{definition}
Note that the second fundamental form and the shape operator are tensors. In view of Proposition \ref{P:equivalentdefinition}, we may write
\begin{align}
D_X Y &=\nabla_X Y+ \II(X,Y) \\
D_X \eta &= -S_\eta X +D^\perp_X \eta
\end{align}

\begin{example}
Given a Riemannian manifold $M$, let $\phi: M\to V$ be an isometric immersion into a pseudo-Euclidean vector space $(V,\bil{,})$. Then $\Omega=d\phi:TM\to V$ is a virtual immersion, with symmetric second fundamental form. On the other hand, for any virtual immersion $\Omega$, the normal part of $d\Omega(X,Y)$ equals $\II(X,Y) -\II(Y,X)$ and, since the tangent part of $d\Omega$ vanishes, it follows that if $\II$ is symmetric then $d\Omega=0$, which implies that locally $\Omega=d\phi$ for some map $\phi:M\to V$. By condition $(a)$ in the definition of virtual immersion, this map must be an isometric immersion.
\end{example}

\begin{proposition}
\label{P:fundamental}
Let $\Omega$ be a virtual immersion of the Riemannian manifold $(M,g)$ with values in $V$. Then the following identities hold:
\begin{enumerate}[a)]
\item Weingarten's equation
 \[ \left< S_\eta(X), Y \right>= \left< \II(X,Y),\eta \right> \]
\item Gauss' equation
\[R(X,Y,Z,W)=\left<\II(Y,W),\II(X,Z)\right>-\left<\II(X,W),\II(Y,Z)\right>\]
\item Ricci's equation
\[  \left< R^\perp(X,Y)\eta,\zeta\right>=-\left<(S_\eta^tS_\zeta -S_\zeta^t S_\eta)X,Y\right>\]
\item Codazzi's equation
\[ \left<(D_X \II) (Y,Z) ,\eta \right>= \left<(D_Y \II) (X,Z) ,\eta \right>. \]
\end{enumerate}
\end{proposition}

\begin{proof}
The proof is the same as in classical case. For sake of completeness, we recall it here.

Fix a point $p$ and let $V=T_pM\oplus \nu_pM$ be the orthogonal splitting into tangent and normal part. Recall that this is possible even though $(V,\bil{,})$ is not Euclidean, because the restriction to $T_pM$ is positive definite. Given vectors $X,Y,Z,W\in T_pM$, extend them locally to vector fields (denoted with the same letters). Differentiating the equation $D_Y Z =\nabla_Y Z+ \II(Y,Z)$ with respect to $X$, one gets
\begin{align*}
D_XD_Y Z =& D_X\left(\nabla_Y Z+ \II(Y,Z)\right)\\
		=&\nabla_X\nabla_Y Z+ \II(X,\nabla_YZ)+D_X(\II(Y,Z)).
\end{align*}
Since the connection $D$ is flat, its curvature vanishes, and one has
\begin{align}
0=&D_{[X,Y]}Z - D_XD_YZ+D_YD_XZ \label{E:GaussCodazzi}\\
=&\big(\nabla_{[X,Y]}Z+\II([X,Y], Z)\big)- \big(\nabla_X\nabla_Y Z+ \II(X,\nabla_YZ)+D_X(\II(Y,Z))\big)\nonumber\\
&+\big(\nabla_Y\nabla_X Z+ \II(Y,\nabla_XZ)+D_Y(\II(X,Z))\big)\nonumber\\
=&R(X,Y)Z- (D_X\II)(Y,Z)+(D_Y\II)(X,Z).\nonumber
\end{align}
Taking the product of both sides of \eqref{E:GaussCodazzi} with $W\in T_pM$, one gets
\begin{align*}
0=&\bil{R(X,Y)Z,W}-\bil{D_X(\II(Y,Z)),W}+\bil{D_Y(\II(X,Z)),W}\\
=&\bil{R(X,Y)Z,W}+\bil{\II(Y,Z),D_XW}-\bil{\II(X,Z),D_YW}\\
=&\bil{R(X,Y)Z,W}+\bil{\II(Y,Z),\II(X,W)}-\bil{\II(X,Z),\II(Y,W)}
\end{align*}
which recovers the Gauss' equation.

On the other hand, taking the product of equation \eqref{E:GaussCodazzi} with $\eta\in \nu_pM$, one obtains
\[
0=\bil{-(D_X\II)(Y,Z)+(D_Y\II)(X,Z),\eta}
\]
which is Codazzi Equation.

Ricci equation is obtained similarly, but starting with equation $D_X \eta = -S_\eta X +D^\perp_X \eta$ instead of $D_X Y =\nabla_X Y+ \II(X,Y)$. Weingarten equation is immediate.

\end{proof}

\section{Virtual immersions on symmetric spaces}\label{S:virtual-imm}

This section is devoted to proving the first part of the main theorem. Namely, given a symmetric space $M$, we show how to produce a virtual immersion $\Omega:TM\to V$ with skew-symmetric second fundamental form.

Since the universal cover $\tilde{M}$ of $M$ is a simply connected symmetric space, by the de Rham decomposition theorem it splits isometrically into irreducible factors, $\tilde{M}=\prod_{i=0}^k\tilde{M}_i$, where $\tilde{M}_0=\R^r$ and none of the other factors is Euclidean. For each $i=0,\ldots k$, choose $p_i\in \tilde{M}_i$, and let $G_i$ be the subgroup of the isometry group of $M$, generated by transvections (i.e. products of two reflections). Then $G_i$ is connected and, by the standard theory of symmetric spaces, it acts transitively on $\tilde{M}_i$. Moreover, $(G_i,H_i)$ is a symmetric pair, where $H_i=(G_i)_{p_i}$. Notice that $G_0=\R^r$, and $H_0=1$.

Let $\pi_i:G_i\to \tilde{M}_i=G_i/H_i$ denote the projection $\pi_i(g):=g\cdot p_i$. Let $\g_i, \h_i$ denote the Lie algebras of $G_i$, $H_i$ respectively, and let $\m_i\subset \g_i$ be a complement of $\h_i$ satisfying $[\m_i,\m_i]\subseteq \h_i$, $[\m_i,\h_i]\subseteq \h_i$. Then the Killing form $B_i$ on $\g_i$ restricts to a negative-definite (resp. positive-definite, zero) symmetric form on $\m_i$ when $\tilde{M}_i$ is of compact (resp. non-compact, Euclidean) type. Moreover, $\m_i$ can be canonically identified with $T_{p_i}\tilde{M}_i$ via $(\pi_i)_*$ and, for $i>0$, the restriction $g_{\tilde{M}}|_{\tilde{M}_i}$ of the metric $g_{\tilde M}$ to $T_{p_i}\tilde{M}_i$ corresponds to $\lambda_i B_i\big|_{\m_i}$ for some negative (resp. positive) value $\lambda_i\in \R$ if $\tilde{M}_i$ is of compact (resp. non-compact) type.

Letting $G=\prod_{i=0}^k G_i$ and $H=\prod_{i=0}^k H_i$, then $(G,H)$ is a symmetric pair, with $G$ acting transitively on $\tilde{M}$ and such that $H=G_{p}$, $p=(p_0,\ldots p_k)$. In particular, $\tilde{M}$ is diffeomorphic to $G/H$, via the map sending $\llbracket g \rrbracket=\llbracket g_0,\ldots g_k \rrbracket \in G/H$ to $g\cdot p=(g_0\cdot p_0,\ldots g_k\cdot p_k)$. Let
\[
\g=\bigoplus_{i=0}^k\g_i,\quad \h=\bigoplus_{i=0}^k\h_i,\quad \m=\bigoplus_{i=0}^k\m_i,
\]
so that $\g=\h\oplus \m$, $[\m,\m]\subseteq \h$ and $[\m, \h]\subseteq \m$. Define $G\times_{H} \m$ as the quotient of $G\times\m$ by the action of $H$ given by $h\cdot (g,X)=(gh^{-1}, \Ad_h X)$, and denote by $\llbracket g,X \rrbracket$ the image of $(g,X)\in G\times\m$ under the quotient map. There is a natural $G$-action on $G\times_{H} \m$, defined by $g'\cdot \llbracket g,X \rrbracket=\llbracket g'g,X \rrbracket$. Extend now the isomorphism
\[
\m=\bigoplus_{i=0}^k\m_i\to \bigoplus_{i=0}^k T_{p_i} \tilde{M}_i=T_{p}\tilde{M}
\]
to the $G$-equivariant bundle isomorphism $G\times_H \m\to T\tilde{M}$ given by $\llbracket g,X\rrbracket\mapsto dg(X)$.

We can now define the virtual immersion $\tilde{\Omega}_0$ on $\tilde{M}$. Endow $\g=\R^r\oplus \bigoplus_{i=1}^k\g_i$ with the nondegenerate symmetric bilinear form
\[
\bil{\,,\,}=g_{\tilde{M}}|_{\R^r}\oplus \bigoplus_{i=1}^k \lambda_i B_i,
\]
and define
\begin{align}
\label{E:naturalgenimm}
\tilde\Omega_0: T\tilde{M}\simeq G\times_{H}\m&\longrightarrow\,\g\\
\llbracket g, X\rrbracket &\longmapsto\,\, \Ad_{g}X \nonumber
\end{align}

\begin{lemma}
\label{L:Omega0}
The $\g$-valued one-form $\tilde{\Omega}_0$ defined in Equation \eqref{E:naturalgenimm} is a virtual immersion. At $\llbracket g\rrbracket \in \tilde{M}$, the tangent and normal spaces are $\Ad_g\m$ and $\Ad_g \h$, respectively. The second fundamental form is skew symmetric, given by 
\[ \II\big(\llbracket g,X\rrbracket,\llbracket g,Y\rrbracket\big)=\Ad_g([X,Y]). \]
\end{lemma}
\begin{proof}
We  begin by showing that condition a) in the definition of virtual immersion holds for $\tilde{\Omega}_0$. By $G$-equivariance it is enough to show that
\[
\tilde{\Omega}_0|_{\llbracket e\rrbracket\times \m}: \llbracket e\rrbracket\times \m\to \g
\]
is an isometric embedding. The embedding is simply the canonical inclusion, therefore given $X,Y\in \m\simeq T_{\llbracket e\rrbracket}\tilde{M}$, and denoting $X_i, Y_i$ the projections of $X,Y$ onto $\m_i\simeq T_{p_i}\tilde{M}_i$, one has
\begin{align*}
\bil{\tilde{\Omega}_0(X),\tilde{\Omega}_0(Y)}=&\bil{X,Y}\\
			=&\bil{X_0,Y_0}+\sum_{i=1}^k \bil{X_i,Y_i}\\
			=&g_{\tilde{M}}(X_0, Y_0)+\sum_{i=1}^k\lambda_i B_i(X_i, Y_i)\\
			=&g_{\tilde{M}}(X_0,Y_0)+\sum_{i=1}^kg_{\tilde{M}}(X_i,Y_i)\\
			=&g_{\tilde{M}}(X,Y).
\end{align*}

It is clear from \eqref{E:naturalgenimm} that the tangent space is $\Ad_g\m$, thus the normal space must be $\Ad_g \h$. 

Let $X\in\g$. Under the identification of $T\tilde{M}$ with $G\times_H \m$ that we are using, the action field $X^*$ is given by 
\[
X^*\llbracket g \rrbracket= \llbracket g,(\Ad_{g^{-1}}X)_\m \rrbracket
.\]
Indeed, $X^*\llbracket g \rrbracket$ is a vector of the form $\llbracket g, v \rrbracket$, with $v=dg^{-1}(X^*\llbracket g \rrbracket)\in \m$. One computes
\begin{align*}
 v= dg^{-1} \left(\left. \frac{d}{dt}\right|_{t=0} \llbracket e^{tX} g \rrbracket \right)& =\left. \frac{d}{dt}\right|_{t=0}\llbracket g^{-1} e^{tX} g\rrbracket\\
& =d\pi_e(\Ad_{g^{-1}}X) \\
&=(\Ad_{g^{-1}}X)_\m,
\end{align*}
where $\pi$ denotes the map $\pi:G\to G/H$.
Given $X,Y\in\g$, we then have
\begin{align*}  
D_{X^*}\tilde{\Omega}_0(Y^*)&=\left.\frac{d}{dt}\right|_{t=0}\tilde{\Omega}_0\llbracket e^{tX}g,(\Ad_{(e^{tX}g)^{-1}}Y)_\m\rrbracket\\
&=\left.\frac{d}{dt}\right|_{t=0} \Ad_{e^{tX}g}(\Ad_{g^{-1}e^{-tX}}Y)_\m)\\
&=\Ad_g\big([\Ad_{g^{-1}}X,(\Ad_{g^{-1}}Y)_\m]  - (\Ad_{g^{-1}}[X,Y])_\m   \big)
\end{align*}

By $G$-equivariance, it is enough to show that,  for every $X,Y\in T_{\llbracket e\rrbracket}\tilde{M}\simeq \m$, we have $d\tilde{\Omega}_0(X^*,Y^*)^T_{ \llbracket e\rrbracket}=0$ and $\II(X,Y)_{ \llbracket e\rrbracket}=[X,Y]$. Plugging $g=e$ in the equation above, and using the fact that $[\m,\m]\subset\h$, we have
\[ D_{X^*}\tilde{\Omega}_0(Y^*)= [X,Y]. \]
The tangent part of this is zero, so that
\[d\tilde{\Omega}_0(X^*,Y^*)^T_{ \llbracket e\rrbracket}= D_{X^*}\tilde{\Omega}_0(Y^*)^T_{ \llbracket e\rrbracket} - D_{Y^*}\tilde{\Omega}_0(X^*)^T_{ \llbracket e\rrbracket} - \tilde{\Omega}_0([X^*,Y^*])_{ \llbracket e\rrbracket}= 0-0-0=0\]
which means that $\tilde{\Omega}_0$ is a virtual immersion.

Moreover, 
 $\II(X,Y)_{ \llbracket e\rrbracket}=D_{X^*}\tilde{\Omega}_0(Y^*)^\perp_{ \llbracket e\rrbracket}= [X,Y]$.
 
\end{proof}

Using the lemma above, we can prove
\begin{lemma}\label{L:full}
The virtual immersion $\tilde{\Omega}_0:T\tilde{M}\to \g$ is full.
\end{lemma}
\begin{proof}
It is enough to prove that $\tilde{\Omega}_0(T_{\tilde{p}}\tilde M)\oplus \textrm{span}\{\II(X,Y)\mid X, Y\in T_{\llbracket e \rrbracket}\tilde M\}=\g$. By  Lemma \ref{L:Omega0},
\[
\tilde{\Omega}_0(T_{\tilde{p}}\tilde M)=\m, \qquad \textrm{span}\{\II(X,Y)\mid X, Y\in T_{\tilde{p}}\tilde M\}=[\m,\m],
\]
therefore this reduces to proving $[\m,\m]=\h$. If not, then there exists a nonzero $H\in \h$ such that $B(H,[X,Y])=0$ for all $X,Y$ in $\m$. By $\Ad$-invariance of the Killing form, this implies $B([H,Y],X)=0$ for all $X,Y\in \m$. Since $[H,Y]\in \bigoplus_{i=1}^t\m_i$ and $B$ is nondegenerate on $\bigoplus_{i=1}^r\g_i$, it follows that $[H,Y]=0$ for every $Y\in \m$. This implies that $\Ad(\exp tH)\in H=G_{\tilde{p}}$ is the identity on $\m=T_{\tilde{p}}\tilde{M}$, which implies $H=0$ hence the contradiction.
\end{proof}

Having defined the virtual immersion $\tilde{\Omega}_0$ on $\tilde{M}$, the goal is now to prove that it descends to a virtual immersion on $M$. This is equivalent to proving that $\tilde{\Omega}_0$ is invariant under the group $\Gamma$ of deck transformations of $\tilde{M}\to M$.

\begin{lemma}\label{L:symm-and-invariance}
Let $\Gamma$ be a discrete subgroup of isometries of $\tilde{M}$ acting freely on $\tilde{M}$. Then the virtual immersion $\tilde{\Omega}_0$ defined above is invariant under $\Gamma$ if and only if $M=\tilde{M}/\Gamma$ is a symmetric space.
\end{lemma}
\begin{proof}
Suppose first that $M$ is a symmetric space, and let $\tau:\tilde{M}\to M$ denote the universal cover of $M$. Then, since the symmetry $s_{\tilde{p}}$ at any $\tilde{p}\in \tilde{M}$ is a lift of the corresponding symmetry $s_p$ at $p=\tau(\tilde{p})\in M$, it follows that for any $\gamma\in \Gamma$, $s_{\tilde{p}}\gamma s_{\tilde{p}}^{-1}$ is a lift of the identity or, in other words, $s_{\tilde{p}}\gamma s_{\tilde{p}}^{-1}\in \Gamma$. Since $M=\tilde{M}/\Gamma$ is a symmetric space and in particular a homogeneous space, by the main theorem in \cite{Wolf} it follows that every element $\gamma\in \Gamma$ is a \emph{Clifford-Wolf translation}, i.e., the displacement function $q\in \tilde{M}\mapsto d(q, \gamma(q))$ is constant. In particular, for any $\tilde{p}\in \tilde{M}$ the isometry $\gamma s_{\tilde{p}}\gamma s_{\tilde p}^{-1}=\gamma \cdot (s_{\tilde{p}}\gamma s_{\tilde p}^{-1})\in \Gamma$ is a Clifford-Wolf translation.

We claim that $\gamma s_{\tilde{p}}\gamma s_{\tilde p}^{-1}$ fixes $\tilde{p}$, which implies that $\gamma s_{\tilde{p}}\gamma s_{\tilde p}^{-1}=id$. In fact, since $\gamma$ is a Clifford-Wolf translation, then $\gamma^{-1}(\tilde p), \tilde p, \gamma (\tilde p)$ all lie on the same geodesic $c(t)$ (cf. \cite[Theorem 1.6]{Ozols}). Parametrize $c(t)$ so that $c(0)=\tilde p$, $c(1)=\gamma(\tilde p)$, $c(-1)=\gamma^{-1}(\tilde p)$. Then, since $s_{\tilde{p}}(\tilde{p})=\tilde{p}$ and $s_{\tilde p}(c(t))=c(-t)$, it follows that
\[
s_{\tilde{p}}\gamma s_{\tilde{p}}^{-1}(\tilde{p})=s_{\tilde{p}}\gamma(\tilde{p})=s_{\tilde{p}}(c(1))=c(-1)=\gamma^{-1}(\tilde{p})
\]
and therefore $\gamma s_{\tilde{p}}\gamma s_{\tilde{p}}^{-1}(\tilde{p})= \tilde{p}$, thus proving the claim.

If follows that $s_{\tilde{p}}\gamma s_{\tilde{p}}^{-1}= \gamma^{-1}$ and therefore, every $\gamma\in \Gamma$ commutes with every transvection. Since $G$ is generated by transvections, then $\Gamma$ commutes with $G$ and thus $\Ad_\gamma$ acts trivially on $\g$ for every $\gamma \in \Gamma$.

Given $\Omega_0: T\tilde{M}=G\times_H\m\to \g$ and fixing $\gamma\in \Gamma$, the map $\Omega_0\circ \gamma: T\tilde{M}=G\times_H\m\to \g$ is given by
\begin{align*}
(\Omega_0\circ \gamma) \llbracket g, X\rrbracket=& \Omega_0\llbracket \gamma g, X\rrbracket= \Ad_{\gamma g}(X)= \Ad_\gamma(\Ad_g X)= \Ad_gX=\Omega_0\llbracket g, X\rrbracket
\end{align*}
and therefore $\Omega_0$ is invariant under $\Gamma$.

On the other hand, suppose now that $\Omega_0$ is invariant under $\Gamma$. Then for every $\gamma\in \Gamma$, $\Ad_\gamma|_{\g}=id$, i.e., $\Gamma$ commutes with $G$ (recall, $G$ is connected). Since $G$ acts transitively on $\tilde{M}$ it follows that every $\gamma\in \Gamma$ is a Clifford-Wolf translation: in fact, for any $\tilde{p}, \tilde{q}\in \tilde{M}$, letting $g\in G$ be such that $g\cdot \tilde{p}=\tilde{q}$, one has
\begin{align*}
d(\tilde{p}, \gamma \tilde{p})=d(g\tilde{p}, g(\gamma \tilde{p}))=d(g\tilde{p}, \gamma( g \tilde{p}))=d(\tilde{q}, \gamma \tilde{q}).
\end{align*}
Moreover, since $G$ is also normalized by the symmetries $s_{\tilde{p}}$ centered at any $\tilde{p}\in \tilde{M}$, it follows that $s_{\tilde{p}}\gamma s_{\tilde{p}}^{-1}$, and thus $\gamma s_{\tilde{p}}\gamma s_{\tilde{p}}^{-1}$, commute with $G$ for any $\gamma\in \Gamma$. In particular $\gamma s_{\tilde{p}}\gamma s_{\tilde{p}}^{-1}$ is again a Clifford-Wolf translation. However, just as before it follows that  $\gamma s_{\tilde{p}}\gamma s_{\tilde{p}}^{-1}$ fixes $\tilde{p}$, and therefore $s_{\tilde{p}}\gamma s_{\tilde{p}}^{-1}=\gamma^{-1}$. In particular, every symmetry $s_{\tilde{p}}$ satisfies $s_{\tilde{p}}\Gamma s_{\tilde{p}}^{-1}=\Gamma$. Therefore, for any point $p=\tau[\tilde{p}]\in M/\Gamma$, one can define a symmetry $s_p:M\to M$ by $s_p[\tilde{q}]=[s_{\tilde{p}}(\tilde{q})]$. In particular, $M$ is a symmetric space.
\end{proof}

\section{Rigidity of virtual immersions with skew-symmetric second fundamental form}\label{S:rigidity}

In this section we prove the second half of the main theorem. Namely, given a minimal virtual immersion $\Omega:TM\to V$ with skew-symmetric second fundamental form, we prove that $M$ is a symmetric space and $\Omega$ is equivalent to the virtual immersion defined in the previous section.

\begin{lemma}
\label{L:locsym}
Let $(M,g)$ be a  Riemannian manifold, and $\Omega$ a $V$-valued virtual immersion with skew-symmetric second fundamental form $\II$. Then:
\begin{enumerate}[a)]
\item $\left< \II(X,Y), \II(Z,W)\right>=\left< R(X,Y)Z, W\right>$.
\item $(D_X \II)(Y,Z)=-R(Y,Z)X$.
\item $\nabla R=0$. In particular, $(M,g)$ is a locally symmetric space.
\end{enumerate}
\end{lemma}

\begin{proof}
\begin{enumerate}[a)]
\item Start with Gauss' equation (see Proposition \ref{P:fundamental}(b)),
\[\left<R(X,Y)Z,W\right>=\left<\II(Y,W),\II(X,Z)\right>-\left<\II(X,W),\II(Y,Z)\right>\]
Applying the first Bianchi identity yields
\[ 0=-2\big(\left<\II(X,Y),\II(Z,W)\right> + \left<\II(Y,Z),\II(X,W)\right>
+ \left<\II(Z,X),\II(Y,W)\right> \big)\]
so that using Gauss' equation one more time we arrive at \[\left< R(X,Y)Z, W\right>=\left< \II(X,Y), \II(Z,W)\right>.\]
\item First we argue that $(D_X \II)(Y,Z)$ is tangent. Indeed, for any normal vector $\eta$, Codazzi's equation (Proposition \ref{P:fundamental}(d)) says that 
\[ \left<(D_X \II) (Y,Z) ,\eta \right>= \left<(D_Y \II) (X,Z) ,\eta \right>. \]
Thus the trilinear map $(X,Y,Z)\mapsto  \left<(D_X \II) (Y,Z) ,\eta \right>$ is symmetric in the first two entries and skew-symmetric in the last two entries, which forces it to vanish.

Next we let $W$ be any tangent vector and compute
\begin{align*}
 \left<(D_X \II) (Y,Z) ,W \right> &=  \left<D_X ( \II(Y,Z)) ,W \right> =  -\left< \II (Y,Z) ,D_XW \right>\\
 &= -\left< \II (Y,Z) ,\II(X,W) \right> = -\left< R(Y,Z)X ,W \right>
\end{align*}
where in the last equality follows we have used part (a).

\item Since the natural connection $D$ on $M\times V$ is flat, it follows that for any vector fields $X,Y,Z,W$, we have
\[ 0= D_X(D_Y(\II(Z,W))) -D_Y(D_X(\II(Z,W))) -D_{[X,Y]}(\II(Z,W)). \]
Fix $p\in M$, and take vector fields such that $[X,Y]=0$ and $\nabla Z=\nabla W=0$ at $p\in M$. Then, evaluating the equation above at $p\in M$, we have
\begin{align*}
0  = & D_X \big( (D_Y\II) (Z,W) +\II(\nabla_Y Z, W) +\II(Z, \nabla_Y W)\big)\\
& - D_Y \big( (D_X\II) (Z,W) +\II(\nabla_X Z, W) +\II(Z, \nabla_X W)\big)  \\
= & D_X (-R(Z,W)Y) + \II(\nabla_X\nabla_Y Z, W) + \II(Z,\nabla_X\nabla_Y W) \\
& -D_Y (-R(Z,W)X) - \II(\nabla_Y\nabla_X Z, W) - \II(Z,\nabla_Y\nabla_X W) \\
= & -(D_X R)(Z,W)Y + (D_Y R)(Z,W)X -\II(R(X,Y)Z,W) \\
&- \II(Z, R(X,Y)W)
\end{align*}
Taking the tangent part yields $(\nabla_X R)(Z,W)Y = (\nabla_Y R)(Z,W)X$. Taking inner product with $T\in T_pM$ we have
\[(\nabla R) (Z,W,Y,T,X)=(\nabla R) (Z,W,X,T,Y),\]
that is, $\nabla R$ is symmetric in the third and fifth entries.
But $\nabla R$ is also skew-symmetric in the third and fourth entries, so that $\nabla R=0$.
\end{enumerate}
\end{proof}

The virtual immersion $\Omega$ on $M$ lifts to a virtual immersion with skew-symmetric second fundamental form $\tilde{\Omega}$ on the universal cover $\tilde{M}$ of $M$. In the following Proposition, we prove that $\tilde{\Omega}$ is equivalent to $\tilde{\Omega}_0$.

\begin{proposition}
\label{P:unique}
Let $(\tilde{M},g_{\tilde{M}})$ be a symmetric space, and let $\Omega_j:T\tilde M\to V_j$, for $j=1,2$ be virtual immersions with skew-symmetric second fundamental forms $\II_j$. Assume $V_1, V_2$ are full. Then $\Omega_1, \Omega_2$ are equivalent.
\end{proposition}
\begin{proof}
Define a connection $\hat{D}$ on the vector bundle $T\tilde M\oplus \wedge^2 T\tilde M$ by
\[ \hat{D}_W (Z, \alpha) =
\big(\nabla_W Z - R(\alpha )W,\,\,\, W\wedge Z+\nabla_W\alpha\big)  \]
Here, for $\alpha=\sum_u X_u\wedge Y_u$, we define $R(\alpha):=\sum_u R(X_u, Y_u)$. Define bundle homomorphisms $\hat\Omega_j :T\tilde M \oplus \wedge^2 T\tilde M \to \tilde M\times V_j$, for $j=1,2$, by
\[ \hat\Omega_j(Z,\alpha) = \Big(p, \Omega_j(Z)+\II_j(\alpha)\Big) \]
for $Z\in T_{\tilde p}\tilde M,\, \alpha=\sum_u X_u\wedge Y_u \in \wedge^2 T_{\tilde p}\tilde M$, and $\II(\alpha)=\sum_u\II(X_u,Y_u)$. By Lemma \ref{L:locsym}(b), given vector fields $Z, W$ and a section $\alpha$ of $\wedge^2 T\tilde M$, we have
\begin{equation}
\label{E:connections}
(D_j)_W \big( \hat\Omega_j(Z,\alpha) \big)= \hat\Omega_j \big(\hat{D}_W (Z,\alpha) \big)
\end{equation}
where $D_j$ denotes the natural flat connection on $\tilde M\times V_j$. This implies that the image of $\hat\Omega_j$ is $D_j$-parallel, and hence, by minimality of $V_j$, that $\hat\Omega_j$ is onto $\tilde M\times V_j$. In particular, for $j=1,2$ the normal space in $V_j$ is spanned by $\II_j(X,Y)$, for $X,Y\in T_{\tilde p} \tilde M$.

Now, we claim that
\[
\ker \hat{\Omega}_1=\ker\hat{\Omega}_2=\big\{ (0,\alpha) \mid\alpha\in\wedge^2T_{\tilde p}\tilde M,\, R(\alpha)=0\big\}.
\]
Indeed, on the one hand if $R(\alpha)=0$ then for every $\beta\in \wedge^2 T_{\tilde p}\tilde M$ one obtains that  $\left<\II_j(\alpha),\II_j(\beta)\right>=\left<R(\alpha),\beta\right>=0$ by Lemma \ref{L:locsym}(a). Since the inner product on $\nu_{\tilde p}\tilde M\subset V_j$ is nondegenerate and the normal space in $V_j$ consists of the elements $\II_j(\beta)$ by the conclusion above, it follows that $\II(\alpha)=0$ and thus $\hat{\Omega}_j(0,\alpha)=0+\II_j(\alpha)$ is zero.

On the other hand, if $\hat{\Omega}_j(Z,\alpha)=0$ then $\Omega_j(Z)=0$ and $\II_j(\alpha)=0$, which implies $Z=0$ and, for every $\beta\in \wedge^2T_{\tilde p}\tilde M$, $0=\left<\II_j(\alpha), \II_j(\beta)\right>=\left<R(\alpha),\beta\right>$ by Lemma \ref{L:locsym}(a). Since the inner product on $\wedge^2T_{\tilde p}\tilde M$ is non-degenerate, it follows that $R(\alpha)=0$ in $\wedge^2 T_{\tilde{p}}\tilde M$, and this ends the proof of the claim.


Since $\hat{\Omega}_i$, $i=1,2$ are both surjective with the same kernel, there is a well-defined bundle isomorphism $L:M\times V_1\to M\times V_2$ by
\[
L \big(\hat\Omega_1(Z,\alpha)\big) = \hat\Omega_2(Z,\alpha)
\]
for $Z\in T_pM$, $\alpha\in \wedge^2T_pM$.

We claim the linear map $L_p=L|_{\{p\}\times V_1}: \{p\}\times V_1\to \{p\}\times V_2$ is independent of $p\in M$. Indeed, given two points $p,q\in M$, choose a curve $\gamma(t)$ in $M$ joining $p$ to $q$. Choose $\hat{D}_1$-parallel vector fields $Z,X_i,Y_i$ along $\gamma(t)$ such that $\hat\Omega_1(Z,\sum X_i\wedge Y_i)$ is constant equal to $v\in V_1$. Then, by \eqref{E:connections}, $\hat D_{\dot\gamma} (Z,\sum X_i\wedge Y_i)\In \ker\hat\Omega_1 $. But by Lemma \ref{L:locsym}(a), $\ker\hat\Omega_1=\ker\hat\Omega_2$. Therefore, again by \eqref{E:connections}, we see that $L(v)$ is constant along $\gamma$, so that $L_p=L_q$.
Calling this one linear map $L$, we have $\hat\Omega_2=L\circ\hat\Omega_1$ by construction. In particular, $\Omega_2=L\circ\Omega_1$, finishing the proof that $\Omega_1$ and $\Omega_2$ are equivalent.

\end{proof}

Piecing all together, we can prove the main Theorem:

\begin{proof}[Proof of the Main Theorem]
Suppose first that $M$ is a symmetric space, and let $\tilde{M}$ be its universal cover. From Lemma \ref{L:Omega0}, there exists a skew-symmetric virtual immersion $\tilde{\Omega}_0:T\tilde{M}\to M\times V$ with $V=\g$. By Lemma \ref{L:symm-and-invariance}, since $M$ is symmetric then $\tilde{\Omega}_0$ is invariant under $\pi_1(M)$ and therefore $\tilde{\Omega}_0$ descends to a skew-symmetric virtual immersion $\Omega:TM\to V$.

Suppose now, on the other hand, that $M$ admits a full, skew-symmetric virtual immersion $\Omega:TM\to  V$. By Lemma \ref{L:locsym} $M$ is locally symmetric, thus the universal cover $\tilde M$ is a symmetric space and $\Omega$ lifts to a skew-symmetric virtual immersion $\tilde{\Omega}: T\tilde{M}\to V$ invariant under the action of $\Gamma=\pi_1(M)$. Since $\tilde{M}$ also admits the virtual immersion $\tilde\Omega_0$, which is full by Lemma \ref{L:full}, it follows from the rigidity Proposition \ref{P:unique} that $\tilde{\Omega}=\tilde{\Omega}_0$, and in particular $\tilde\Omega_0$ is invariant under the action of $\Gamma$. By Lemma \ref{L:symm-and-invariance}, it follows that $M$ is a symmetric space.
\end{proof}

\bibliographystyle{alpha}

\end{document}